\documentclass[11pt]{amsart}

\usepackage[margin=1in]{geometry}
\usepackage{amsmath, amssymb, amsfonts, amsthm, mathtools, mathrsfs}
\usepackage{graphicx}
\usepackage[hidelinks,bookmarks=false]{hyperref}
\usepackage{enumitem}

\theoremstyle{plain}
\newtheorem{theorem}{Theorem}[section]
\newtheorem{lemma}[theorem]{Lemma}
\newtheorem{proposition}[theorem]{Proposition}

\theoremstyle{definition}

\theoremstyle{remark}
\newtheorem{remark}[theorem]{Remark}

\allowdisplaybreaks

\newcommand{\R}{\mathbb{R}}

\newcommand{\Z}{\mathbb{Z}}

\newcommand{\lan}{\langle}
\newcommand{\ran}{\rangle}

\newcommand{\gammat}{\gamma_t}

\newcommand{\Lsharp}{L^{\#}_{\gamma}}
\newcommand{\Gtau}[1][]{G^{\tau}_{#1}}
\newcommand{\id}{\mathrm{id}}
\newcommand{\eps}{\epsilon}
\newcommand{\dmuLt}{d\mu_{L,t}}
\newcommand{\dotp}{\dot{p}}

\title[First variation of flat traces]{First variation of flat traces on negatively curved surfaces}
\author{Hy P.\ G.\ Lam}
\address{Department of Mathematical Sciences, Worcester Polytechnic Institute, Worcester, MA 01609}
\email{hlam@wpi.edu}
\address{Department of Mathematics, Northwestern University, Evanston, IL 60208}
\email{hylam2023@northwestern.edu, hylam.math@gmail.com}

\thanks{Research partially supported by NSF RTG grant DMS-1502632.}

\date{}

\begin{document}

\begin{abstract}
For a closed negatively curved surface $(X,g)$ the flat trace of the geodesic Koopman operators $V_g^\tau f=f\circ G_g^\tau$ is the periodic orbit distribution
\[
\mathrm{Tr}^{\flat} V_{g}(\tau)=\sum_{\gamma}\frac{L_\gamma^{\#}}{\lvert\det(I-P_\gamma)\rvert}\,\delta(\tau-L_\gamma),
\qquad \tau>0,
\]
supported on the length spectrum and weighted by the linearized Poincar\'e maps $P_\gamma$.
For a smooth family of negatively curved metrics $g_t$ we compute the first variation $\partial_t\vert_{0}\,\mathrm{Tr}^{\flat} V_{g_t}$ as a distribution.
At an isolated length $\ell$ the leading singularity is a multiple of $\delta'(\tau-\ell)$, and its coefficient is an explicit linear functional of the length variations $\dot L_{\gamma^m}$ of the closed geodesics with $L_{\gamma^m}=\ell$.
This transport coefficient forces the marked lengths to be locally constant along any deformation with constant flat trace.
As an application, if $\mathrm{Tr}^{\flat} V_{g_t}=\mathrm{Tr}^{\flat} V_{g_0}$ for all $t$ then $g_t$ is isometric to $g_0$ for all $t$.
Together with Sunada-type constructions of non isometric pairs with equal flat traces, this shows that the flat trace is globally non-unique yet locally complete along smooth families.
\end{abstract}

\maketitle

\section{Introduction}

The dynamical side of spectral geometry rests on periodic orbit expansions attached to the geodesic flow.
On a closed negatively curved surface, the flat trace of the geodesic Koopman operators $V_g^\tau f=f\circ G_g^\tau$ is the distribution supported on the length spectrum
\begin{equation}\label{flatformula}
\mathrm{Tr}^{\flat} V_{g}(\tau)=\sum_{\gamma}\frac{L_\gamma^{\#}}{\lvert\det(I-P_\gamma)\rvert}\,\delta(\tau-L_\gamma).
\end{equation}
Its Laplace transform gives the logarithmic derivative of the Guillemin Ruelle dynamical zeta built with the same weights, so equality of flat traces implies equality of the corresponding zetas.

Let $X$ be a closed oriented surface and let $g_t$ be a $C^\infty$ family of negatively curved metrics on $X$ for $t\in(-\eps,\eps)$.
Write $G_t^\tau$ for the geodesic flow on $S^*_{g_t}X$, write $V_t^\tau f=f\circ G_t^\tau$ for the Koopman operator on $L^2(S^*_{g_t}X,\dmuLt)$, and let $\mathrm{Tr}^{\flat}V_t$ denote its flat trace distribution on $(0,\infty)$.
Here $\dmuLt$ is the Liouville measure induced on $S^*_{g_t}X$ by the Hamiltonian $H_t(x,\xi)=\frac12\lvert\xi_x\rvert_{g_t}^2$.

The main analytic input is the distributional first variation of $\mathrm{Tr}^{\flat}V_t$ under a metric deformation.
At an isolated length $\ell>0$, the distribution $\partial_t\vert_{0}\,\mathrm{Tr}^{\flat}V_t$ has a leading singularity of the form $\mathcal T(\ell)\,\delta'(\tau-\ell)$.
The coefficient $\mathcal T(\ell)$ is an explicit transport functional of the marked length variations.
Proposition~\ref{mostsingularterm} gives
\begin{equation}\label{eq:intro_transport}
\mathcal T(\ell)=-\sum_{\substack{\gamma,m\\ L_{\gamma^m}=\ell}}\frac{L_\gamma^{\#}}{\lvert\det(I-P_\gamma^m)\rvert}\,\dot L_{\gamma^m}.
\end{equation}
In particular, constancy of the flat trace forces $\mathcal T(\ell)=0$ for every isolated $\ell$.

Once the transport coefficients vanish, the persistence of hyperbolic periodic orbits under smooth perturbation implies $\dot L_{\gamma}=0$ for every closed geodesic.
Equivalently the marked length spectrum is locally constant in $t$.
At this point one may invoke the marked length spectrum rigidity theorem of Croke~\cite{Cr} and Otal~\cite{Ot} to conclude that $g_t$ is isometric to $g_0$ for all $t$.
We continue instead with a direct deformation theoretic argument in the same dynamical framework.
After the $\delta'$ constraint fixes all periods, the Guillemin-Kazhdan variational identity yields vanishing of the period integrals of $\dot p=\partial_t\vert_{0}\,\lvert\xi\rvert_{g_t}$.
Liv\v{s}ic theory converts these vanishing period integrals into a cohomological equation on $S^*_{g_0}X$, and the Guillemin-Kazhdan operator calculus identifies $\dot p$ with a Lie derivative on the base.
This completes a proof of infinitesimal triviality which does not appeal to marked length spectrum rigidity and which is expected to be robust for Anosov flows in higher dimension.

\begin{remark}
In dimension $n\ge 3$, the Lefschetz fixed point reduction and the Abel-type extraction of the $\delta'$-term remain formally identical, but the clean codimension and the harmonic bookkeeping change. Concretely, the circle decomposition in Appendix \ref{sec:gkcalculus} is replaced by the $SO(n-1)$ representation theory of spherical harmonics on $\mathbb S^{n-2}$, so the mode reduction and coercive estimates must be organized by irreducible $SO(n-1)$ types rather than Fourier modes of $SO(2)$. In this regime, the marked length spectrum rigidity input of Croke and Otal is not available, so the direct flat trace variational approach is the natural substitute once the $SO(n-1)$-mode calculus is implemented.
\end{remark}

As an application of this analysis we obtain the following pathwise rigidity statement.

\begin{theorem}\label{theoremA}
If $\mathrm{Tr}^{\flat}V_{g_t}=\mathrm{Tr}^{\flat}V_{g_0}$ as distributions on $(0,\infty)$ for all $t$, then there exists a smooth one parameter family of diffeomorphisms $\{\varphi_t\}$ of $X$ with $\varphi_0=\id$ such that
\[
g_t=\varphi_t^*g_0 \qquad \text{for all } t\in(-\eps,\eps).
\]
Equivalently, $(X,g_t)$ is isometric to $(X,g_0)$ for every $t$.
\end{theorem}

Sunada type constructions produce non isometric pairs with equal flat traces at fixed metrics \cite{La}.
Together with those examples, Theorem~\ref{theoremA} shows that the flat trace is globally non unique yet locally complete along smooth paths inside the space of negatively curved metrics.

\begin{remark}\label{intro}
The rigidity statement in Theorem~\ref{theoremA} is weaker than marked length spectrum rigidity, since it assumes a smooth deformation.
The point of the present paper is the distributional first variation formula for the flat trace, in particular the transport coefficient in Proposition~\ref{mostsingularterm}.
This coefficient does not seem to appear in the literature in this explicit form, and it is the mechanism that converts constancy of a zeta type invariant into infinitesimal constraints on the marked lengths.
The appendices collect auxiliary calculations and are included for completeness.
\end{remark}

The paper is organized as follows.
Section~\ref{sec:firstvariationcalculus} recalls the flat trace distribution of the geodesic Koopman operator and derives its first variation.
Section~\ref{variationgeodesiclength} records the metric variation of closed geodesic lengths and of the Hamiltonian $p_t(x,\xi)=\lvert\xi\rvert_{g_t}$.
Section~\ref{variationofflattrace} expresses the $\delta'$-coefficient explicitly in terms of the metric perturbation $h = \dot{g}$ and extracts the resulting linear relation among the $h$-period integrals.
Section~\ref{sec:trivialdeformation} combines the $\delta'$ constraint with Liv\v{s}ic theory and the Guillemin-Kazhdan $SO(2)$ calculus, summarized in Appendix~\ref{sec:gkcalculus}, to conclude infinitesimal triviality and then integrate it in $t$ to obtain Theorem~\ref{theoremA}.
Appendix~\ref{sec:lowertermscomputation} contains the verification that the flat trace formula has no additional lower order singular terms beyond the Lefschetz coefficients.

\subsection*{Acknowledgements}
This article refines and unifies material that appeared in the author's doctoral dissertation at Northwestern University.
The author gratefully acknowledges the guidance of the late Steve Zelditch.

\section{Flat trace, clean fixed sets, and first variation}\label{sec:firstvariationcalculus}

\subsection{Kernel and flat trace as a pullback and pushforward}

For $\tau > 0 $, the Schwartz kernel of $V^\tau_g$ is the $\delta$-section
\begin{equation*}
K^\tau_t (\zeta, \zeta') = \delta( \zeta' - G^\tau_t \zeta), \qquad \zeta, \zeta' \in S^*_{g_t}X,
\end{equation*}
understood as a conormal distribution to $\mathrm{Graph}(G^\tau_t)$. The flat trace is the pullback by the diagonal $\iota: \zeta \mapsto(\zeta, \zeta) $ followed by fiber integration
\begin{equation}\label{flattrace}
    \mathrm{Tr}^{\flat} V_{t}(\tau)   = \int_{S^*_{g_t}X} \iota^*K^\tau_t(\zeta) \dmuLt (\zeta) = \int_{S^*_{g_t}X} \delta(\zeta - G^\tau_t \zeta) \dmuLt(\zeta).
\end{equation}

In negative curvature, the fixed set of $\Gtau[t]$ is clean precisely when $\tau$ equals the period of a (possibly iterated) closed geodesic.

\subsection{Normal form near a closed orbit}

Fix a primitive closed geodesic $\gamma$ of $(X, g_0)$ with prime length $\Lsharp > 0$, linearized Poincar\'{e} map $P_\gamma$ and let $\gamma^m$ denote its $m$-iterate of length $L_{\gamma^m} = m\Lsharp$. For $\tau$ near $L_{\gamma^m}$, the fixed set of $\Gtau[0]$ is the circle $\gamma$ (as a subset of $S^*_{g_0}X$). Choose symplectic coordinates $(s,y; \sigma, \eta)$ near $\gamma$ so that $s \in \R\backslash \Lsharp \Z$ parametrizes the orbit, $(y, \eta) \in \R^2$ are transverse canonical coordinates, and the flow has the normal form
\begin{equation*}
    \Gtau[0] : (s,y; \sigma, \eta) \mapsto (s + \tau, \Phi_\tau(y, \eta); \sigma, \Psi_\tau(y, \eta) ),
\end{equation*}
$ (y, \eta) = (0,0)$ on $\gamma$, and $d\Phi_{L_{\gamma^m}} $ conjugates to $P^m_\gamma$ on the transversal.

\smallskip

As a conormal distribution,  $K^\tau_0$ may be represented by an oscillatory integral
\begin{equation*}
K^\tau_0 (\zeta, \zeta') = (2\pi)^{-N} \int e^{i \lan \theta, \zeta' - \Gtau[0]\zeta \ran } a_0 (\zeta, \zeta', \tau, \theta) d\theta
\end{equation*}
with a classical amplitude $a_0$. Pulling back by the diagonal and integrating over $\zeta$ (cf. \eqref{flattraceintime}) gives
\begin{equation}\label{flatint}
    \mathrm{Tr}^{\flat} V_{0}(\tau) = (2\pi)^{-N} \int e^{i \lan \theta, \zeta - \Gtau[0]\zeta \ran } a_0(\zeta, \zeta, \tau , \sigma) d\sigma d\zeta .
\end{equation}

\subsection{Stationary phase with a clean one dimensional critical set}
The phase $\Phi_\tau(\zeta, \theta) : = \lan \theta, \zeta - \Gtau[0]\zeta\ran$ has critical set
$$\mathcal{C}_\tau  = \{ (\zeta, \theta): \zeta \in \mathrm{Fix}(\Gtau[0]), \theta \in N^*_\zeta\mathrm{Fix}(\Gtau[0])\},
$$
which, for $\tau$ near $L_{\gamma^m}$, is a vector bundle over the circle $\gamma$. The clean stationary-phase theorem (\cite{DG}) yields
\begin{equation}\label{flatasympexp}
\mathrm{Tr}^{\flat} V_{0}(\tau)  \sim \sum_{m \ge 1} \sum_\gamma \int_\R e^{i\sigma (\tau - L_{\gamma^m}) } b_{\gamma,m}(\sigma) d\sigma \quad + \quad C^\infty(\tau),
\end{equation}
where, crucially, the leading term of the classical symbol $b_{\gamma,m}$ is
\begin{equation}\label{leadingterm}
    b^{(0)}_{\gamma,m} = (2\pi)^{-1} \frac{\Lsharp}{|\det ( I - P^m_\gamma) |}.
\end{equation}
By evaluating the $\sigma$-integral,
we have
\begin{equation}\label{flattracemicrolocal}
    \mathrm{Tr}^{\flat} V_{t}(\tau) = \sum_{m \ge  1 } \sum_\gamma \frac{\Lsharp}{|\det ( I - P^m_\gamma) |}  \delta(\tau - L_{\gamma^m}) + \, (\text{lower order terms at $\tau = L_{\gamma^m}$} )   + C^\infty(\tau),
\end{equation}
which is the standard flat-trace formula whereby the subprincipal corrections in \eqref{flatasympexp} yields only derivatives of $\delta$ of order $\le 0 $, i.e. multiples of $\delta$ because the critical set is 1-dimensional.

\begin{remark}\label{nohigherderivativesatfixedt}
For a Fourier integral operator on an $n$-dimensional manifold whose fixed set at time $\tau$ is clean of dimension $d$, the clean stationary phase theorem shows that the flat trace singularity at $\tau$ has order $(n-d-1)/2$.  In our setting, $n=\dim(S^*_gX)=3$ and the relevant fixed sets are the closed geodesics, so $d=1$.  Hence, at a fixed metric $g_t$, the singularity at $\tau=L_{\gamma^m}(t)$ is of order $0$, and no derivatives $\delta^{(k)}$ with $k\ge1$ appear in the flat trace at fixed $t$.  By contrast, differentiating with respect to the parameter $t$ may produce $\delta'$-terms through transport of the singular support (cf.\ Proposition~\ref{mostsingularterm}).
\end{remark}

The absence of any further singular contributions at a fixed period (i.e.\ the fact that \eqref{flattracemicrolocal} has no additional ``lower order'' $\delta$-terms) is proved in Appendix~\ref{sec:lowertermscomputation}.

\subsection[Parameter dependent phase and t variation]{Parameter dependent phase and $t$ variation of the clean contribution}
Let $g_t$ be a small deformation. Let $\chi \in C^\infty_c$ localize near a single pair $(\gamma, m)$. Microlocally near $\gamma$,

\begin{equation}\label{flattraceintime}
    \mathrm{Tr}^{\flat} V_{t}(\tau) = \int\int e^{i \Phi_t(\zeta, \theta, \tau)} a_t(\zeta, \theta, \tau) d\theta d\zeta,  \quad \text{with} ~~ \Phi_t = \lan \theta, \zeta- \Gtau[t] \zeta  \ran.
\end{equation}
As before, we apply the method of joint-normal stationary phase to the clean critical submanifold
\begin{equation}\label{cleanjointnormal}
C_{t, \tau = L_{\gamma^m}(t)}
=
\{(\zeta,\theta):\ \zeta\in\gamma_t,\ \theta\in N_\zeta^*\gamma_t\}.
\end{equation}
to obtain
\begin{equation}\label{tracemicrolocal}
    \mathrm{Tr}^{\flat} V_{t}(\tau)  \sim \int_\R e^{i\sigma(\tau - L_{\gamma^m}(t)) } \, b_{\gamma, m}(t,\sigma) d\sigma  \quad + \quad C^\infty(\tau, t),
\end{equation}
with (cf. \eqref{leadingterm})
$$
b_{\gamma,m} (t, \sigma) \sim \sum_{j \geq 0} b^{(j)}_{\gamma, m} (t) \langle \sigma\rangle ^{-j }, \quad b^{(0)}_{\gamma,m} (t) = (2\pi)^{-1} \frac{\Lsharp(t)}{|\det ( I - P^m_{\gamma,t}) |}
$$
For a test function $\psi \in C^\infty_c( (0,\infty))$,
we pair \eqref{tracemicrolocal} with $\psi$ after differentiating at $t = 0$ 
to obtain
\begin{equation}\label{tracepair}
    \langle  \partial_t\big|_0 \mathrm{Tr}^{\flat} V_{t}, \,  \psi \rangle = \int_\R e^{-i\sigma L_{\gamma^m}} \check{\psi}(\sigma)\left( \partial_t b_{\gamma,m }(t, \sigma)\big|_{t = 0}   - i\sigma b_{\gamma,m}(0,\sigma) (m\dot{L}^\#_\gamma)  \right) d\sigma
\end{equation}
where $\check{\psi}(\sigma) = \int_{\R_{>0}} e^{i\sigma \tau}\psi(\tau) d\tau$ and $m\dot{L}^\#_\gamma = \partial_t (L_{\gamma^m}(t)) \big|_{t= 0 } $. By the asymptotic expansion
$$b_{\gamma, m}(0,\sigma) = (2\pi)^{-1} \frac{\Lsharp}{|\det (I - P^m_\gamma)|}
 + O  (\langle \sigma \rangle^{-1}) $$ and the Fourier identities
$$
\int e^{i\sigma (\tau - \ell)} d\sigma = 2\pi \delta(\tau - \ell), \quad \int e^{i\sigma(\tau - \ell) }(i\sigma) d\sigma = 2\pi \delta'(\tau - \ell),
$$
summing over all $(\gamma, m)$ provides, in the sense of distribution in $\tau$,
\begin{equation}\label{partialttrace}
\partial_t\big|_0 \mathrm{Tr}^{\flat} V_{t}(\tau ) = -\sum_{\gamma}\sum_{m \ge 1} \frac{L_\gamma^\#}{|\det(I - P^m_\gamma)|}(m\dot{L}^\#_\gamma)\delta'(\tau - L_{\gamma^m}) + \sum_{\gamma,m}\dot{A}_{\gamma, m} \delta(\tau -L_{\gamma^m})
\end{equation}
where $\dot{A}_{\gamma,m}$  are smooth amplitude-variation weights stemming from $\partial_t b_{\gamma, m}(t , \sigma)\big|_{t= 0}$, in particular, from $\partial_t\det(I - P^m_\gamma)$, the Liouville density, etc. The first term is the most singular because it is the only source of $\delta'$.

\subsection{Isolation of the most singular terms and the weighted identities}

Let $\ell > 0$ be such that the set $\{ (\gamma,m) \, :   \,  L_{\gamma^m} = \ell \}$ is finite (this is true in negative curvature). Take $\psi \in C^\infty_c((0,\infty)) $ supported in a small neighborhood of $\ell$ and normalized by $\psi(\ell) = 0, \quad \psi'(\ell) = 1$. We again pair \eqref{partialttrace} with $\psi$. Since $ \langle \delta(\tau - \ell) , \psi \rangle = \psi(\ell) =  0 $ and $\langle  \delta'(\tau - \ell), \psi \rangle = -\psi' (\ell)  = - 1$, we get from \eqref{partialttrace}
\begin{equation}\label{flattderivative}
    \big\langle \partial_t\big|_0 \mathrm{Tr}^{\flat} V_{t} , \, \psi \big\rangle   = \sum_{\substack{\gamma,m \ge 1 \\ L_{\gamma^m} = \ell}} \frac{L^\#_\gamma}{|\det(I - P^m_\gamma)|} m\dot{L}_\gamma^\#.
\end{equation}
Therefore,
\begin{proposition}\label{mostsingularterm}
    At each isolated length $\ell >0$, the length-transport coefficient of $\delta'(\tau - \ell)$, in $\partial_t\big|_0 \mathrm{Tr}^{\flat} V_{t}(\tau)$ equals
$$
\mathcal{T}(\ell): = -  \sum_{\substack{\gamma, m \\ L_{\gamma^m} = \ell  }}  \frac{L^\#_\gamma}{|\det(I - P^m_\gamma)|} \dot{L}_{\gamma^m}.
$$
In particular, if $\mathrm{Tr}^{\flat} V_{t}$ is constant in $t$, then for every $\ell$, $\mathcal{T}(\ell) = 0$.
\end{proposition}
\smallskip

\begin{remark}\label{transportcoeff}
\textup{ $\mathcal{T}(\ell)$ captures that the $\delta'$ arises from transport of the singular support as the lengths move as a consequence of the metric deformation. In fact, no other term in \eqref{partialttrace} contributes to $\delta'$ as the amplitude variations always yield $\delta$ (or smoother) at $\tau = \ell$.
Secondly,  $\mathcal{T}(\ell) = 0$ due to constancy is the exact $\delta'$-constraint needed in $\S$ \ref{variationofflattrace}. Namely, after identifying $m\dot{L}_\gamma^\# = \dot{L}_{\gamma^m} = \frac{1}{2}\int_{\gamma^m}h (T,T) ds$, it becomes a linear relation among the integrals of $h$ over orbits of length $\ell$. Lastly, when the flat trace is constant for all $t$, the support of $\mathrm{Tr}^{\flat} V_{t}$ is independent of $t$. Indeed, by structural stability and discreteness of the length spectrum, this forces each $L_{[\alpha]}(t)$ to be locally constant, hence constant, which is what we use subsequently . }
\end{remark}


\bigskip

\section[Variations of L and p]{The variations $\dot L_\gamma$ and $\dot p$ in terms of $\dot g$}\label{variationgeodesiclength}

Throughout this section, we fix $t = 0 $ and write
\begin{equation}\label{eq: gtsecondorderexp}
    g_t = g + th + O (t^2), \quad h \in \Gamma(S^2 T^*X).
\end{equation}

\subsection{Variation of length for a closed geodesic}

Fix a free homotopy class $[\alpha]$ of loops in $X$, and for each $t$ let $\gamma_t:S^1\to X$ be the unique closed $g_t$-geodesic in $[\alpha]$.  We parametrize $\gamma_t$ with constant $g_t$-speed on $[0,1]$, i.e.
\[
|\dot\gamma_t(u)|_{g_t}\equiv L_{[\alpha]}(t),\qquad
L_{[\alpha]}(t)=\int_0^1|\dot\gamma_t(u)|_{g_t}\,du.
\]
Write $\gamma:=\gamma_0$, $L:=L_{[\alpha]}(0)$ and set
\[
T:=\frac{\dot\gamma}{|\dot\gamma|_g}\in TX\big|_{\gamma},\qquad ds:=|\dot\gamma|_g\,du=L\,du,
\]
so that $T$ is the $g$-unit tangent vector field and $s$ is the $g$-arclength parameter.

Let $V:=\partial_t\gamma_t|_{t=0}$ be the variational vector field along $\gamma$.  Differentiating
\[
L_{[\alpha]}(t)=\int_0^1\sqrt{g_t(\dot\gamma_t,\dot\gamma_t)}\,du
\]
at $t=0$ gives
\begin{align}
\dot L_{[\alpha]}(0)
&=\frac12\int_0^1\frac{h(\dot\gamma,\dot\gamma)+2g(\nabla_u V,\dot\gamma)}{|\dot\gamma|_g}\,du\notag\\
&=\frac12\int_0^1 h(T,T)\,ds+\int_0^1 g(\nabla_uV,T)\,du.\label{eq:first_var_length_pre}
\end{align}
The second term vanishes by an integration by parts. Indeed, since
\[
\frac{d}{du}g(V,T)=g(\nabla_uV,T)+g(V,\nabla_uT),
\]
and $\gamma$ is a $g$-geodesic, $\nabla_TT=0$, hence $\nabla_uT=(ds/du)\nabla_TT=0$.  Therefore
$g(\nabla_uV,T)=\frac{d}{du}g(V,T)$, and by periodicity on $S^1$ we obtain
\[
\int_0^1 g(\nabla_uV,T)\,du=g(V,T)\Big|_{u=0}^{u=1}=0.
\]
Substituting into \eqref{eq:first_var_length_pre} yields the standard first variation formula
\begin{equation}\label{eq:variation_length}
\dot L_{[\alpha]}(0)=\frac12\int_{\gamma}h(T,T)\,ds.
\end{equation}
Equivalently, for any closed geodesic $\gamma$ (and its $m$-iterate $\gamma^m$),
\begin{equation}\label{eq:variation_length_iterate}
\dot L_{\gamma^m}=\frac12\int_{\gamma^m}h(T,T)\,ds.
\end{equation}

\subsection[Variation of the Hamiltonian]{Variation of the Hamiltonian on $T^*X$}\label{variationofhamiltonian}

Let $H_t:T^*X\to\mathbb{R}$ be the kinetic energy Hamiltonian
\[
H_t(x,\xi):=\frac12|\xi|^2_{g_t}=\frac12\,g_t^{ab}(x)\,\xi_a\xi_b,
\]
and write
\[
p_t(x,\xi):=|\xi|_{g_t}=\sqrt{2H_t(x,\xi)}.
\]
Differentiating $g_t^{ab}g_{t,bc}=\delta^a_c$ gives $\dot g^{ab}=-h^{ab}$, where $h^{ab}:=g^{ac}g^{bd}h_{cd}$.  Therefore
\[
\dot H=\frac12\,\dot g^{ab}\xi_a\xi_b=-\frac12\,h^{ab}\xi_a\xi_b.
\]
Since $p_t=(2H_t)^{1/2}$, we obtain
\begin{equation}\label{eq:dot_p_general}
\dot p=\frac{\dot H}{p}=-\frac{1}{2p}\,h^{ab}\xi_a\xi_b.
\end{equation}
In particular, on the unit cosphere bundle $S^*_gX=\{p=1\}$,
\begin{equation}\label{eq:dot_p_unit}
\dot p\big|_{S^*_gX}=-\frac12\,h^{ab}\xi_a\xi_b=-\frac12\,h(T,T),
\end{equation}
where $T=\xi^\sharp\in S_gX$ is the corresponding unit tangent vector.  Combining \eqref{eq:dot_p_unit} with \eqref{eq:variation_length_iterate} gives, for every closed geodesic $\gamma$,
\begin{equation}\label{eq:dotL_dotp_relation}
\dot L_\gamma=-\int_\gamma \dot p\,ds.
\end{equation}
Moreover, $\dot p$ is an even function under the fiber flip $\xi\mapsto-\xi$, and in the $SO(2)$-Fourier decomposition on $S_gX$ its Fourier support is contained in modes $0,\pm2$ (cf.\ Appendix~\ref{sec:gkcalculus}).

\medskip

\section{Variation of the flat trace with explicit coefficients}\label{variationofflattrace}

This section records the explicit form of the $\delta'$-constraint in geometric terms.
For each fixed $t$ the flat trace of the Koopman operator of the geodesic flow is given by the (exact) Lefschetz/Guillemin formula
\begin{equation}\label{eq:flat_trace_exact}
\mathrm{Tr}^{\flat}V_{t}(\tau)
=
\sum_{\gamma\in\mathcal{P}(g_t)}\ \sum_{m\geq 1}
\frac{L_{\gamma}^{\#}(t)}{\left|\det\!\left(I-P_{\gamma,t}^{\,m}\right)\right|}\,
\delta\!\big(\tau-L_{\gamma^m}(t)\big)
\;+\;C^\infty(\tau),
\end{equation}
where $\mathcal{P}(g_t)$ denotes the set of primitive closed $g_t$-geodesics, $L_{\gamma^m}(t)=mL^\#_\gamma(t)$ is the length of the $m$-iterate, and $P_{\gamma,t}$ is the linearized Poincar\'e map (cf.\ Appendix~\ref{sec:lowertermscomputation}).

Differentiating \eqref{eq:flat_trace_exact} at $t=0$ gives a distribution supported on the length spectrum.  The only possible $\delta'$-terms come from differentiating the moving supports $\tau=L_{\gamma^m}(t)$, and one obtains (cf.\ Proposition~\ref{mostsingularterm})
\begin{equation}\label{eq:delta_prime_constraint}
\partial_t\Big|_{t=0}\mathrm{Tr}^{\flat}V_{t}(\tau)
=
-\sum_{\gamma,m}\frac{L^\#_\gamma}{\left|\det\!\left(I-P_{\gamma}^{\,m}\right)\right|}\,
\dot L_{\gamma^m}\,\delta'\!\big(\tau-L_{\gamma^m}\big)\;+\;\text{\rm(only $\delta$ and smoother)}.
\end{equation}
Combining \eqref{eq:delta_prime_constraint} with the length variation formula \eqref{eq:variation_length_iterate} yields an explicit expression for the $\delta'$-coefficient in terms of $h=\dot g$:
\begin{equation}\label{eq:delta_prime_in_terms_of_h}
\text{Coeff}_{\delta'(\tau-L_{\gamma^m})}\Big(\partial_t\mathrm{Tr}^{\flat}V_{t}\Big|_{t=0}\Big)
=
-\frac{L^\#_\gamma}{\left|\det\!\left(I-P_{\gamma}^{\,m}\right)\right|}\,
\frac12\int_{\gamma^m} h(T,T)\,ds.
\end{equation}
In particular, if $\mathrm{Tr}^{\flat}V_{t}$ is constant in $t$, then for every $\ell$ in the length spectrum,
\begin{equation}\label{eq:linear_relation_h_integrals}
\sum_{\substack{\gamma,m\\L_{\gamma^m}=\ell}}
\frac{L^\#_\gamma}{\left|\det\!\left(I-P_{\gamma}^{\,m}\right)\right|}\,
\int_{\gamma^m} h(T,T)\,ds
=0.
\end{equation}
When the length spectrum is simple this already implies $\int_{\gamma}h(T,T)\,ds=0$ for each closed $\gamma$; in general we will use instead the stronger support argument in \S\ref{constantlengthfromdelta'} to deduce that all lengths are constant in $t$.


\section[Triviality from vanishing of delta prime]{Triviality from the vanishing of $\delta'$}\label{sec:trivialdeformation}

\subsection{Guillemin-Kazhdan variational identity}

Let $p_t$ be as in $\S $ \ref{variationofhamiltonian} and $y_t \subset \{ p_t =1 \} $ be the closed characteristic corresponding to $\gammat$. If $L_{\gamma_t}$ is constant in $t$ (see $\S$ \ref{constantlengthfromdelta'}), then Guillemin-Kazhdan's strip argument yields
\begin{equation}\label{strip}
\int_\gamma \dot{p} ds = 0
    \end{equation}

\begin{lemma}[Guillemin-Kazhdan]\label{GKidentity}
Suppose $t\mapsto y_t$ is a $C^\infty$ family of closed orbits of the Hamiltonian flows of $p_t$ on $\{p_t=1\}$, all with the same period $L$.  Then
\[
\int_{y_0}\dot p\,ds=0,
\]
where $ds$ denotes the Hamiltonian time parameter along $y_0$.
\end{lemma}

\begin{proof}
Let $\alpha$ be the canonical $1$-form on $T^*X$ and $\omega=d\alpha$ the canonical symplectic form.  Let $H_t$ be the Hamiltonian vector field of $p_t$, characterized by
\[
\iota_{H_t}\omega = dp_t.
\]
Choose a parametrization $\Phi:[0,\varepsilon]\times(\mathbb{R}/L\mathbb{Z})\to T^*X$ of the cylinder swept out by the closed orbits,
\[
\Phi(t,s):=y_t(s),
\qquad \partial_s\Phi(t,s)=H_t(\Phi(t,s)).
\]
Since $p_t\circ \Phi\equiv 1$, differentiating in $t$ gives
\[
0=\partial_t(p_t\circ\Phi)=\dot p_t(\Phi)+dp_t(\partial_t\Phi).
\]
Using $\omega(H_t,\cdot)=dp_t(\cdot)$, we compute
\[
\omega(\partial_t\Phi,\partial_s\Phi)
=\omega(\partial_t\Phi,H_t)
=-\,\omega(H_t,\partial_t\Phi)
=-\,dp_t(\partial_t\Phi)
=\dot p_t(\Phi).
\]
Hence $\Phi^*\omega=\dot p_t(\Phi)\,dt\wedge ds$.  By Stokes' theorem,
\begin{align*}
\int_0^\varepsilon\int_0^L \dot p_t(y_t(s))\,ds\,dt
&=\int_{[0,\varepsilon]\times(\mathbb{R}/L\mathbb{Z})}\Phi^*\omega
=\int_{[0,\varepsilon]\times(\mathbb{R}/L\mathbb{Z})}d(\Phi^*\alpha)\\
&=\int_{\partial([0,\varepsilon]\times(\mathbb{R}/L\mathbb{Z}))}\Phi^*\alpha
=\int_{y_\varepsilon}\alpha-\int_{y_0}\alpha.
\end{align*}
Since $p_t$ is homogeneous of degree $1$ in $\xi$, Euler's identity gives $\alpha(H_t)=p_t$, and therefore
\[
\int_{y_t}\alpha=\int_0^L \alpha(H_t)\,ds=\int_0^L p_t\,ds=L,
\]
which is independent of $t$ by hypothesis.  Thus the right-hand side vanishes, and differentiating the left-hand side at $\varepsilon=0$ yields $\int_{y_0}\dot p\,ds=0$.
\end{proof}

\subsection[Deducing constant lengths from delta prime]{Deducing constant lengths from the $\delta'$ vanishing}\label{constantlengthfromdelta'}

Assume that $\mathrm{Tr}^{\flat}V_{g_t}=\mathrm{Tr}^{\flat}V_{g_0}$ as distributions on $(0,\infty)$ for all $t\in(-\varepsilon,\varepsilon)$.  By the exact flat trace formula \eqref{eq:flat_trace_exact}, for each fixed $t$ we have
\[
\mathrm{Tr}^{\flat}V_{g_t}(\tau)
=
\sum_{\gamma\in\mathcal{P}(g_t)}\ \sum_{m\ge1}
\frac{L^\#_\gamma(t)}{\left|\det\!\left(I-P_{\gamma,t}^{\,m}\right)\right|}\,
\delta\!\big(\tau-L_{\gamma^m}(t)\big)
\;+\;C^\infty(\tau).
\]
Every coefficient in this expansion is strictly positive, hence the singular support is exactly the set of periods of closed orbits (the length spectrum with iterates):
\[
\mathrm{sing\,supp}\big(\mathrm{Tr}^{\flat}V_{g_t}\big)=\{L_{\gamma^m}(t):\gamma\in\mathcal{P}(g_t),\ m\ge1\}.
\]
Since $\mathrm{Tr}^{\flat}V_{g_t}$ is independent of $t$, its singular support is independent of $t$ as a subset of $(0,\infty)$.  Therefore the length spectrum of $g_t$ (including iterates) is locally constant as a set.

Fix a free homotopy class $[\alpha]$.  For each $t$ there is a unique closed $g_t$-geodesic $\gamma_{[\alpha]}(t)$ in this class, and by structural stability of Anosov flows the corresponding periodic orbit persists and varies continuously with $t$.  In particular, the length function $t\mapsto L_{[\alpha]}(t)$ is continuous.  But by the previous paragraph, $L_{[\alpha]}(t)$ takes values in the fixed set $\mathrm{sing\,supp}(\mathrm{Tr}^{\flat}V_{g_0})$.  This set is discrete: for an Anosov flow the set of periods of periodic orbits is a discrete subset of $(0,\infty)$ (see, for instance, \cite[Ch.\,18]{KH}).  Hence a continuous map into this set is locally constant, and since $(-\varepsilon,\varepsilon)$ is connected it is constant on the whole interval.  Therefore
\begin{equation}\label{eq:lengths_constant}
L_{[\alpha]}(t)\equiv L_{[\alpha]}(0)\qquad\text{for all free homotopy classes }[\alpha].
\end{equation}
In particular, for every closed geodesic $\gamma$ and every $m\ge1$,
\begin{equation}\label{eq:dotL_zero}
\dot L_{\gamma^m}=0.
\end{equation}
By \eqref{eq:variation_length_iterate}, this is equivalent to the vanishing of the $h$-period integrals:
\begin{equation}\label{eq:h_period_integrals_zero}
\int_{\gamma^m} h(T,T)\,ds=0\qquad\text{for every closed geodesic }\gamma^m.
\end{equation}

\subsection{Proof of Theorem \ref{theoremA}}

\begin{proof}
Fix $t_0\in(-\varepsilon,\varepsilon)$ and set $g:=g_{t_0}$.  Let $p_t(x,\xi)=|\xi|_{g_t}$ and let $X_{g}$ denote the generator of the $g$-geodesic flow on the unit cosphere bundle $S^*_gX=\{p_{t_0}=1\}$.  By \S\ref{constantlengthfromdelta'}, every closed $g_t$-geodesic has $t$-independent length; equivalently, for each periodic orbit $y_{t}$ of the Hamiltonian flow of $p_t$ on $\{p_t=1\}$, its period is independent of $t$.

Let $\dot p:=\partial_t p_t|_{t=t_0}$, viewed as a smooth function on $S^*_gX$.  Applying Lemma~\ref{GKidentity} to the shifted family $t\mapsto p_{t_0+t}$ shows that for every closed orbit $y_0$ of $X_g$,
\begin{equation}\label{eq:livsic_obstruction}
\int_{y_0}\dot p\,ds=0.
\end{equation}
Since $g$ has negative curvature, the flow of $X_g$ is Anosov.  Therefore Liv\v{s}ic's theorem applies: the vanishing of all periodic orbit integrals \eqref{eq:livsic_obstruction} implies that $\dot p$ is a coboundary, i.e.\ there exists a (H\"older) function $u$ on $S^*_gX$ such that
\begin{equation}\label{eq:cohomological_equation}
X_g u=\dot p
\end{equation}
(for general Anosov flows, smooth regularity of the Liv\v{s}ic coboundary follows from \cite{LMM} and \cite{Journe}. In our setting, smoothness follows independently from the $SO(2)$-mode reduction below).

Now, $\dot p$ is even under the flip $\xi\mapsto-\xi$ and has $SO(2)$-Fourier support in modes $0,\pm2$ (cf.\ \eqref{eq:dot_p_unit}).  The Guillemin-Kazhdan $SO(2)$-calculus on the coframe bundle (Appendix~\ref{sec:gkcalculus}, cf.\ \cite{GK}) upgrades the Liv\v{s}ic solution: one can choose $u\in C^\infty(S^*_gX)$ solving \eqref{eq:cohomological_equation}, and moreover $u$ is fiber-linear (its Fourier support is contained in modes $\pm1$).  Concretely, there is a smooth vector field $v$ on $X$ such that
\begin{equation}\label{eq:u_fiber_linear_main}
u(x,\xi)=\langle \xi, v(x)\rangle_g\qquad\text{for }(x,\xi)\in S^*_gX.
\end{equation}

To convert \eqref{eq:cohomological_equation} into a statement about the metric variation, identify $\xi\in S^*_gX$ with its dual unit tangent vector $T:=\xi^\sharp\in S_gX$.  Then $u=g(v,T)$.  Along any $g$-geodesic, $\nabla_TT=0$, and therefore
\[
X_g u
=
\frac{d}{ds}g(v,T)
=
g(\nabla_Tv,T)
=
\frac12(\mathcal{L}_v g)(T,T).
\]
On the other hand, \eqref{eq:dot_p_unit} at time $t=t_0$ reads $\dot p=-\frac12(\partial_t g_t|_{t=t_0})(T,T)$.  Comparing with \eqref{eq:cohomological_equation} gives
\[
(\mathcal{L}_v g)(T,T)=-(\partial_t g_t|_{t=t_0})(T,T)\qquad\text{for all }T\in S_gX,
\]
hence
\begin{equation}\label{eq:metric_variation_is_lie}
\partial_t g_t\big|_{t=t_0}=-\,\mathcal{L}_v g_{t_0}.
\end{equation}

Applying the same argument at each $t_0$ produces a smooth time-dependent vector field $v_t$ on $X$ such that $\partial_t g_t=-\mathcal{L}_{v_t}g_t$ for all $t$ in a possibly smaller interval.  Let $\phi_t$ be the flow of $v_t$, i.e.\ the unique solution of
\[
\frac{d}{dt}\phi_t = v_t\circ\phi_t,\qquad \phi_0=\mathrm{id}.
\]
Then
\[
\frac{d}{dt}\big(\phi_t^*g_t\big)
=
\phi_t^*\big(\partial_t g_t+\mathcal{L}_{v_t}g_t\big)
=
0,
\]
so $\phi_t^*g_t=g_0$ for all $t$.  Setting $\varphi_t:=\phi_t^{-1}$ we obtain $g_t=\varphi_t^*g_0$, as claimed.
\end{proof}

\section{Appendix}


\subsection[Lower order terms]{Explicit computation of the lower order terms in \eqref{flattracemicrolocal}}

\label{sec:lowertermscomputation}

We work at a fixed metric $g$ of negative curvature.  Recall the definition of the flat trace (cf.\ \eqref{flattraceintime} with $t$ fixed):
\begin{equation}\label{eq:flat_trace_def_app}
\mathrm{Tr}^\flat V(\tau)
=\int_{S^*_gX}\delta\!\big(\zeta-\Gtau\zeta\big)\,d\mu_L(\zeta),
\qquad \tau>0,
\end{equation}
where $\delta(\zeta-\Gtau\zeta)$ is the Dirac distribution on the fixed point set of $\Gtau$ (interpreted via the pullback to the diagonal in the sense of clean intersection).

Fix a primitive closed geodesic $\gamma$ of prime length $\Lsharp$, and its $m$-iterate $\gamma^m$ of length $L_{\gamma^m}=m\Lsharp$.  Choose a flow box $U$ around $\gamma$ and a smooth diffeomorphism
\begin{equation}\label{eq:flow_box}
\Psi:\ \mathbb{S}^1_{\Lsharp}\times B_\rho(0)\subset \mathbb{S}^1_{\Lsharp}\times\mathbb{R}^2 \longrightarrow U\subset S^*_gX,
\qquad (s,y)\longmapsto \Psi(s,y),
\end{equation}
such that $\Psi(s,0)\in\gamma$ parametrizes $\gamma$ by arclength $s$ and the geodesic flow has the exact normal form
\begin{equation}\label{eq:flow_box_dynamics}
\Gtau\big(\Psi(s,y)\big)=\Psi\big(s+\tau,\ \Phi_\tau(y)\big),
\end{equation}
for a smooth family of local diffeomorphisms $\Phi_\tau$ of $B_\rho(0)$ with $\Phi_{0}=\mathrm{id}$ and
\begin{equation}\label{eq:poincare}
D\Phi_{m\Lsharp}(0)=P_\gamma^{\,m}.
\end{equation}
Let $\chi\in C_c^\infty(U)$ be a cutoff which equals $1$ in a smaller neighborhood of $\gamma$ and set $\chi^\Psi:=\chi\circ\Psi$.  The localized flat trace is the distribution
\begin{equation}\label{eq:Tgamma_def}
\mathrm{Tr}^\flat V_\gamma(\tau):=\int_{S^*_gX}\chi(\zeta)\,\delta\!\big(\zeta-\Gtau\zeta\big)\,d\mu_L(\zeta),
\qquad \mathrm{Tr}^\flat V=\sum_\gamma \mathrm{Tr}^\flat V_\gamma + C^\infty(\tau).
\end{equation}

\smallskip

\noindent
\textbf{Reduction to a $\delta$-calculus on $\mathbb{S}^1_{\Lsharp}\times\mathbb{R}^2$.}
Write the Liouville density in flow box coordinates as
\begin{equation}\label{eq:mu_pullback}
\Psi^*(d\mu_L)=\rho(s,y)\,ds\,dy,
\qquad \rho\in C^\infty,\ \rho>0.
\end{equation}
Since $d\mu_L$ restricts to arclength on the closed orbit $\gamma$ and $s$ is arclength, we may (and do) normalize $\Psi$ so that
\begin{equation}\label{eq:rho_normalization}
\rho(s,0)\equiv 1\quad\text{for all }s\in\mathbb{S}^1_{\Lsharp}.
\end{equation}
Pair $\mathrm{Tr}^\flat V_\gamma$ against an arbitrary test function $\varphi\in C_c^\infty((0,\infty))$.  By definition,
\begin{align}
\langle \mathrm{Tr}^\flat V_\gamma,\varphi\rangle
&=\int_0^\infty \varphi(\tau)\int_{S^*_gX}\chi(\zeta)\,\delta\!\big(\zeta-\Gtau\zeta\big)\,d\mu_L(\zeta)\,d\tau \notag\\
&=\int_0^\infty\!\!\int_{\mathbb{S}^1_{\Lsharp}}\!\!\int_{\mathbb{R}^2}
\varphi(\tau)\,\chi^\Psi(s,y)\,\delta\!\Big(\Psi(s,y)-\Psi(s+\tau,\Phi_\tau(y))\Big)\,\rho(s,y)\,dy\,ds\,d\tau.
\label{eq:pairing_start}
\end{align}
Because $\Psi$ is a diffeomorphism, the distribution $\delta(\Psi(s,y)-\Psi(s',y'))$ is the pullback of the delta on the diagonal under $(\Psi\times\Psi)$, and the Jacobian factors cancel against $\Psi^*(d\mu_L)$ in \eqref{eq:pairing_start}.  Consequently, in the variables $(s,y)$ the delta constraint is exactly the simultaneous constraint
\begin{equation}\label{eq:delta_constraints}
s+\tau\equiv s \ \ (\mathrm{mod}\ \Lsharp),\qquad \Phi_\tau(y)=y,
\end{equation}
and we may rewrite \eqref{eq:pairing_start} as the $\delta$-pairing for the map
\(
F(s,y,\tau):=(s-(s+\tau),\,y-\Phi_\tau(y))
\)
on $\mathbb{S}^1_{\Lsharp}\times\mathbb{R}^2\times\mathbb{R}$:
\begin{equation}\label{eq:pairing_deltaF}
\langle \mathrm{Tr}^\flat V_\gamma,\varphi\rangle
=
\int_0^\infty\!\!\int_{\mathbb{S}^1_{\Lsharp}}\!\!\int_{\mathbb{R}^2}
\varphi(\tau)\,\chi^\Psi(s,y)\,\rho(s,y)\,
\delta_{\mathbb{S}^1_{\Lsharp}}\!\big(s-(s+\tau)\big)\,
\delta_{\mathbb{R}^2}\!\big(y-\Phi_\tau(y)\big)\,dy\,ds\,d\tau.
\end{equation}

\smallskip

\noindent
\textbf{The circle delta and the Dirac comb in $\tau$.}
Let $\delta_{\mathbb{S}^1_{\Lsharp}}$ denote the delta distribution on the circle of length $\Lsharp$.  Its pullback to $\mathbb{R}$ is the Dirac comb
\begin{equation}\label{eq:dirac_comb}
\delta_{\mathbb{S}^1_{\Lsharp}}(u)=\sum_{k\in\mathbb{Z}}\delta(u-k\Lsharp)\qquad\text{in }\mathcal{D}'(\mathbb{R}).
\end{equation}
We have
\begin{equation}\label{eq:circle_integral}
\int_{\mathbb{S}^1_{\Lsharp}}\delta_{\mathbb{S}^1_{\Lsharp}}\!\big(s-(s+\tau)\big)\,ds
=\int_{\mathbb{S}^1_{\Lsharp}}\delta_{\mathbb{S}^1_{\Lsharp}}(-\tau)\,ds
=\Lsharp\sum_{k\in\mathbb{Z}}\delta(\tau-k\Lsharp).
\end{equation}
Substituting \eqref{eq:circle_integral} into \eqref{eq:pairing_deltaF} reduces the pairing to
\begin{equation}\label{eq:pairing_after_s}
\langle \mathrm{Tr}^\flat V_\gamma,\varphi\rangle
=
\int_0^\infty \varphi(\tau)\,
\Big(\Lsharp\sum_{k\in\mathbb{Z}}\delta(\tau-k\Lsharp)\Big)\,
\Big[\int_{\mathbb{R}^2} A(\tau,y)\,\delta\!\big(y-\Phi_\tau(y)\big)\,dy\Big]\,d\tau,
\end{equation}
where
\begin{equation}\label{eq:A_def}
A(\tau,y):=\int_{\mathbb{S}^1_{\Lsharp}}\chi^\Psi(s,y)\,\rho(s,y)\,ds.
\end{equation}
Because $\chi^\Psi\equiv 1$ near $y=0$ and \eqref{eq:rho_normalization} holds, we have
\begin{equation}\label{eq:A_at_0}
A(\tau,0)=\int_{\mathbb{S}^1_{\Lsharp}}1\,ds=\Lsharp
\qquad\text{for $\tau$ in a neighborhood of $m\Lsharp$.}
\end{equation}

\smallskip

\noindent
\textbf{The transversal delta and the Poincar\'e determinant.}
Fix $m\ge 1$ and localize in $\tau$ to a small neighborhood of $m\Lsharp$ so that $\gamma^m$ is the unique fixed orbit in the support of $\chi$.  Set
\begin{equation}\label{eq:Ftau_def}
F_\tau(y):=y-\Phi_\tau(y).
\end{equation}
At $\tau=m\Lsharp$, $F_{m\Lsharp}(0)=0$ and, by \eqref{eq:poincare},
\begin{equation}\label{eq:DF}
DF_{m\Lsharp}(0)=I-D\Phi_{m\Lsharp}(0)=I-P_\gamma^{\,m}.
\end{equation}
Negative curvature implies $\gamma$ is hyperbolic, hence $1$ is not an eigenvalue of $P_\gamma^{\,m}$ and $\det(I-P_\gamma^{\,m})\neq 0$.  By the inverse function theorem, $F_{m\Lsharp}$ is a diffeomorphism from a neighborhood of $0$ onto a neighborhood of $0$, and the distribution $\delta(F_{m\Lsharp}(y))$ is computed by a change of variables: for any $\psi\in C_c^\infty(\mathbb{R}^2)$ supported sufficiently close to $0$,
\begin{align}
\int_{\mathbb{R}^2}\psi(y)\,\delta\!\big(F_{m\Lsharp}(y)\big)\,dy
&=\int_{\mathbb{R}^2}\psi\big(F_{m\Lsharp}^{-1}(z)\big)\,\delta(z)\,
\big|\det DF_{m\Lsharp}(F_{m\Lsharp}^{-1}(z))\big|^{-1}\,dz \notag\\
&=\psi(0)\,\big|\det DF_{m\Lsharp}(0)\big|^{-1}
=\psi(0)\,\big|\det(I-P_\gamma^{\,m})\big|^{-1}.
\label{eq:delta_change_of_vars}
\end{align}
Applying \eqref{eq:delta_change_of_vars} to the inner bracket in \eqref{eq:pairing_after_s} with $\psi(y)=A(\tau,y)$ and using \eqref{eq:A_at_0} yields, near $\tau=m\Lsharp$,
\begin{equation}\label{eq:transversal_eval}
\int_{\mathbb{R}^2}A(\tau,y)\,\delta\!\big(y-\Phi_{m\Lsharp}(y)\big)\,dy
=\frac{A(m\Lsharp,0)}{|\det(I-P_\gamma^{\,m})|}
=\frac{\Lsharp}{|\det(I-P_\gamma^{\,m})|}.
\end{equation}
Substituting \eqref{eq:transversal_eval} into \eqref{eq:pairing_after_s} and restricting to $\tau>0$ gives
\begin{equation}\label{eq:Tgamma_pairing_final}
\langle \mathrm{Tr}^\flat V_\gamma, \, \varphi\rangle
=\sum_{m\ge 1}\frac{\Lsharp}{|\det(I-P_\gamma^{\,m})|}\,\varphi(m\Lsharp),
\end{equation}
hence, as distributions on $(0,\infty)$,
\begin{equation}\label{eq:Tgamma_distribution}
\mathrm{Tr}^\flat V_\gamma(\tau)=\sum_{m\ge 1}\frac{\Lsharp}{|\det(I-P_\gamma^{\,m})|}\,\delta(\tau-m\Lsharp)
\quad\text{microlocally near }\tau=m\Lsharp.
\end{equation}

\smallskip

\noindent
\textbf{Conclusion.}
Summing \eqref{eq:Tgamma_distribution} over all primitive closed geodesics and adding the $C^\infty$ contribution from $\tau$ away from the length spectrum yields the sharp Lefschetz flat-trace formula
\begin{equation}\label{eq:flat_trace_no_lower}
\mathrm{Tr}^\flat V(\tau)
=
\sum_{\gamma}\sum_{m\ge 1}\frac{\Lsharp}{|\det(I-P_\gamma^{\,m})|}\,\delta(\tau-L_{\gamma^m})
\;+\;C^\infty(\tau).
\end{equation}
In particular, the parenthetical term ``(lower order terms at $\tau=L_{\gamma^m}$)'' in \eqref{flattracemicrolocal} vanishes identically in the Lefschetz (hyperbolic) setting: there are no additional singular contributions supported at $\tau=L_{\gamma^m}$ beyond the Dirac masses already displayed.


\subsection[SO(2) calculus and mode reduction]{The $SO(2)$ Fourier calculus on $S^*_gX$ and the $\pm 1$ mode reduction}\label{sec:gkcalculus}

Let $(X,g)$ be a closed oriented surface with Gaussian curvature $K<0$, and identify $S^*_gX\simeq SX$ via the musical isomorphism.  Denote by $\pi:SX\to X$ the bundle projection.  We use the standard canonical coframing on $SX$:
there exist smooth $1$-forms $(\alpha,\beta,\psi)$ and smooth vector fields $(X,X_\perp,V)$ on $SX$ uniquely specified by
\begin{align}
    &\alpha(X)=1,  \ \beta(X)=0,\ \psi(X)=0;\quad
\alpha(X_\perp)=0,\ \beta(X_\perp)=1,\ \psi(X_\perp)=0; \notag \\
    &\alpha(V)=\beta(V)=0,  \psi(V)=1,     \label{eq:dual_frame}
\end{align}
together with the structure equations
\begin{equation}\label{eq:structure_eqs}
d\alpha=\psi\wedge\beta,\qquad
d\beta=-\psi\wedge\alpha,\qquad
d\psi=K\,\alpha\wedge\beta.
\end{equation}
The Liouville volume form is
\begin{equation}\label{eq:liouville_volume}
d\mu=\alpha\wedge d\alpha=\alpha\wedge\psi\wedge\beta.
\end{equation}
The geodesic flow generator is $X$ (the Reeb field of $\alpha$), $V$ generates the right $SO(2)$-action (rotation of the fiber angle), and $  X_\perp =   [V,X]$ is the horizontal rotation by $\pi/2$.

\smallskip

\noindent\textbf{Commutators.}
For any $1$-form $\omega$ and vector fields $Y,Z$,
\begin{equation}\label{eq:domega_comm}
d\omega(Y,Z)=Y(\omega(Z))-Z(\omega(Y))-\omega([Y,Z]).
\end{equation}
Applying \eqref{eq:domega_comm} to $\omega=\alpha,\beta,\psi$ and the pairs $(X,V)$, $(X_\perp,V)$, $(X,X_\perp)$, using \eqref{eq:dual_frame}-\eqref{eq:structure_eqs}, gives:
\begin{align*}
d\alpha(X,V)&=(\psi\wedge\beta)(X,V)=\psi(X)\beta(V)-\psi(V)\beta(X)=0-1\cdot 0=0
\ \Rightarrow\ \alpha([X,V])=0,\\
d\beta(X,V)&=(-\psi\wedge\alpha)(X,V)=-(\psi(X)\alpha(V)-\psi(V)\alpha(X))=-(0-1\cdot 1)=1
\ \Rightarrow\ \beta([X,V])= -1,\\
d\psi(X,V)&=(K\alpha\wedge\beta)(X,V)=K(\alpha(X)\beta(V)-\alpha(V)\beta(X))=0
\ \Rightarrow\ \psi([X,V])=0,
\end{align*}
hence $[X,V]= -X_\perp$.  Similarly, evaluating at $(X_\perp,V)$ yields $[V,X_\perp]=-X$.  Finally, at $(X,X_\perp)$,
\[
d\psi(X,X_\perp)=K\alpha\wedge\beta(X,X_\perp)=K,
\qquad
d\psi(X,X_\perp)=-\psi([X,X_\perp]),
\]
so $\psi([X,X_\perp])=-K$, and since $\alpha([X,X_\perp])=\beta([X,X_\perp])=0$ (from $d\alpha(X,X_\perp)=d\beta(X,X_\perp)=0$), we obtain
\begin{equation}\label{eq:commutators}
[V,X]=X_\perp,\qquad [V,X_\perp]=-X,\qquad [X,X_\perp]=-K\,V.
\end{equation}

\smallskip

\noindent\textbf{Skew-adjointness.}
We claim that $X,X_\perp,V$ are divergence-free with respect to $d\mu$, hence skew-adjoint on $L^2(SX,d\mu)$.  By Cartan's formula $L_Y\omega=i_Y d\omega+d(i_Y\omega)$ and \eqref{eq:structure_eqs}, we have
\[
L_X\alpha=i_X(\psi\wedge\beta)+d(\alpha(X))=\psi(X)\beta-\beta(X)\psi+0=0,
\qquad L_X d\alpha=d(L_X\alpha)=0,
\]
thus $L_X(\alpha\wedge d\alpha)=0$.  Next,
\[
L_{X_\perp}\alpha=i_{X_\perp}(\psi\wedge\beta)=\psi(X_\perp)\beta-\beta(X_\perp)\psi=-\psi,
\qquad
L_{X_\perp}d\alpha=d(-\psi)=-d\psi=-K\,\alpha\wedge\beta,
\]
and therefore
\[
L_{X_\perp}(\alpha\wedge d\alpha)=(-\psi)\wedge d\alpha+\alpha\wedge(-K\alpha\wedge\beta)=0.
\]
Finally,
\[
L_V\alpha=i_V(\psi\wedge\beta)=\psi(V)\beta-\beta(V)\psi=\beta,\qquad
L_Vd\alpha=d\beta=-\psi\wedge\alpha,
\]
so
\[
L_V(\alpha\wedge d\alpha)=\beta\wedge d\alpha+\alpha\wedge(-\psi\wedge\alpha)=0.
\]
Hence $L_Xd\mu=L_{X_\perp}d\mu=L_Vd\mu=0$, and by integration by parts,
\begin{equation}\label{eq:skew_adjoint}
\langle Xu,v\rangle=-\langle u,Xv\rangle,\qquad
\langle X_\perp u,v\rangle=-\langle u,X_\perp v\rangle,\qquad
\langle Vu,v\rangle=-\langle u,Vv\rangle,
\end{equation}
for all smooth $u,v$ (and by density for all $H^1$ functions).

\smallskip

\noindent\textbf{Fourier decomposition and the operators $\eta^\pm$.}
Let $V$ be the infinitesimal generator of the right $SO(2)$-action; in a fiber angle coordinate $\theta$ one has $V=\partial_\theta$, hence $V$ is skew-adjoint and its spectrum is $\{im:\ m\in\mathbb{Z}\}$.  Define
\[
\mathcal{H}_m:=\{u\in L^2(SX):\ Vu=imu\ \text{in }\mathcal{D}'(SX)\},
\qquad
L^2(SX)=\widehat{\bigoplus}_{m\in\mathbb{Z}}\mathcal{H}_m,
\qquad u=\sum_{m\in\mathbb{Z}}u_m.
\]
Introduce the raising/lowering operators
\begin{equation}\label{eq:eta_pm_def}
\eta^+:=\tfrac12(X-iX_\perp),\qquad \eta^-:=\tfrac12(X+iX_\perp),
\qquad X=\eta^++\eta^-.
\end{equation}
By \eqref{eq:skew_adjoint}, $(\eta^+)^*=-\eta^-$ and $(\eta^-)^*=-\eta^+$.  Moreover, by \eqref{eq:commutators},
\begin{equation}\label{eq:eta_comm}
[\eta^-,\eta^+]=-\tfrac{i}{2}[X,X_\perp]=-\tfrac{i}{2}(-K V)=\tfrac{i}{2}\,K\,V,
\qquad
[V,\eta^\pm]=\pm i\,\eta^\pm,
\end{equation}
so $\eta^\pm:\mathcal{H}_m\to \mathcal{H}_{m\pm 1}$.

\smallskip

\noindent\textbf{Parity reduction for even data.}
Let $\mathcal{A}:SX\to SX$ be the flip $\mathcal{A}(x,v)=(x,-v)$.  Then $\mathcal{A}$ conjugates the flow to its inverse:
\begin{equation}\label{eq:flip_conjugacy}
\Gtau\circ \mathcal{A}=\mathcal{A}\circ \Gtau[-1],\qquad \tau\in\mathbb{R},
\end{equation}
hence $\mathcal{A}_*X=-X$, while $\mathcal{A}$ is a fiber rotation by $\pi$ and therefore commutes with $V$.  In particular, if $f$ is even under the flip, $f\circ\mathcal{A}=f$, and $Xu=f$, then
\[
X(u\circ\mathcal{A})=(\mathcal{A}_*X)u\circ\mathcal{A}=-(Xu)\circ\mathcal{A}=-f,
\]
so $X(u+u\circ\mathcal{A})=0$.  Since the geodesic flow on a negatively curved surface is ergodic with respect to $d\mu$, any $L^2$ function annihilated by $X$ is constant a.e.; for continuous functions this implies global constancy.  Subtracting half this constant from $u$, we may assume
\begin{equation}\label{eq:u_odd}
u\circ\mathcal{A}=-u,
\end{equation}
i.e.\ $u(x,\theta+\pi)=-u(x,\theta)$ in an angle coordinate.  Consequently, all even Fourier modes vanish:
\begin{equation}\label{eq:odd_modes}
u_m=0\quad\text{for every even }m.
\end{equation}

\smallskip

\noindent\textbf{Mode-by-mode form of the Liv\v{s}ic equation.}
Let $u\in C^1(SX)$ solve
\begin{equation}\label{eq:livsic_eq}
Xu=f,
\end{equation}
with $f\in C^\infty(SX)$.  Decomposing $u=\sum u_m$ and $f=\sum f_m$ and using $X=\eta^++\eta^-$ and $\eta^\pm:\mathcal{H}_m\to\mathcal{H}_{m\pm1}$ gives, by projection onto $\mathcal{H}_m$,
\begin{equation}\label{eq:mode_equation}
f_m=\eta^+u_{m-1}+\eta^-u_{m+1},\qquad m\in\mathbb{Z}.
\end{equation}
In our application $f=\dotp|_{SX}=-\tfrac12 h(v,v)$, where $h=\dot g$ is a smooth symmetric $2$-tensor on $X$ and
$v\in S_xX$ is the unit direction.  To determine the $SO(2)$-Fourier support of $\dot p$, fix a local oriented
$g$-orthonormal frame $(e_1,e_2)$ on $X$ and write the fiber angle coordinate $\theta$ so that
\[
v=\cos\theta\,e_1+\sin\theta\,e_2 .
\]
Set $h_{ij}(x):=h_x(e_i,e_j)$.
Then, using $\cos^2\theta=\frac12(1+\cos2\theta)$, $\sin^2\theta=\frac12(1-\cos2\theta)$, and
$\sin\theta\cos\theta=\frac12\sin2\theta$, we compute
\begin{align*}
h(v,v)
&=h_{11}\cos^2\theta+2h_{12}\sin\theta\cos\theta+h_{22}\sin^2\theta\\
&=\tfrac12(h_{11}+h_{22})
+\tfrac12(h_{11}-h_{22})\cos2\theta
+h_{12}\sin2\theta.
\end{align*}
Equivalently, writing $\cos2\theta=\frac12(e^{2i\theta}+e^{-2i\theta})$ and
$\sin2\theta=\frac{1}{2i}(e^{2i\theta}-e^{-2i\theta})$, we obtain
\begin{equation}\label{eq:dotp_modes}
\dot p(x,\theta)
= -\frac14\,(h_{11}+h_{22})(x)
-\frac18\Big((h_{11}-h_{22})(x)-2i\,h_{12}(x)\Big)e^{2i\theta}
-\frac18\Big((h_{11}-h_{22})(x)+2i\,h_{12}(x)\Big)e^{-2i\theta}.
\end{equation}
In particular,
\[
\dot p\in \mathcal{H}_{-2}\oplus\mathcal{H}_0\oplus\mathcal{H}_2,
\]
and $\dot p$ is even under the flip $\mathcal{A}(x,v)=(x,-v)$.

\smallskip

\noindent\textbf{Energy-curvature identity.}
For $w\in C^\infty(SX)$, using $(\eta^\pm)^*=-\eta^\mp$ and \eqref{eq:eta_comm},
\begin{align}
\|\eta^+w\|_{L^2}^2-\|\eta^-w\|_{L^2}^2
&=\langle w,(\eta^+)^*\eta^+w\rangle-\langle w,(\eta^-)^*\eta^-w\rangle \notag\\
&=\langle w,-\eta^-\eta^+w\rangle-\langle w,-\eta^+\eta^-w\rangle
=\langle w,[\eta^+,\eta^-]w\rangle
=-\langle w,[\eta^-,\eta^+]w\rangle \notag\\
&=-\Big\langle w,\tfrac{i}{2}K Vw\Big\rangle
=-\tfrac{i}{2}\int_{SX}K\,(Vw)\,\overline{w}\,d\mu.
\label{eq:energy_identity}
\end{align}
If $w\in\mathcal{H}_m$ so that $Vw=imw$, then \eqref{eq:energy_identity} becomes
\begin{equation}\label{eq:energy_identity_mode}
\|\eta^+w\|^2-\|\eta^-w\|^2=\frac{m}{2}\int_{SX}K\,|w|^2\,d\mu.
\end{equation}
Since $K\le -\kappa_0<0$, for $m>0$ we obtain the coercive estimate
\begin{equation}\label{eq:coercive_estimate_correct}
\|\eta^+w\|^2 \le \|\eta^-w\|^2 - \frac{\kappa_0\,m}{2}\,\|w\|^2.
\end{equation}

\smallskip

\noindent\textbf{The $\pm 1$-mode reduction.}
Assume $u\in C^1(SX)$ satisfies $Xu=f$ with $f$ even and supported in modes $0,\pm 2$, and normalize $u$ to be odd as in \eqref{eq:u_odd}, so $u$ has only odd modes.  Then $f_m=0$ for all odd $m$ and \eqref{eq:mode_equation} yields, for all $|m|\ge 3$,
\begin{equation}\label{eq:tail_relation}
\eta^+u_{m-1}+\eta^-u_{m+1}=0.
\end{equation}
Taking $L^2$ norms and using orthogonality of Fourier modes implies
\begin{equation}\label{eq:norm_relation}
\|\eta^-u_{m+1}\|=\|\eta^+u_{m-1}\|\qquad (|m|\ge 3),
\end{equation}
and therefore, for $m\ge 3$,
\begin{equation}\label{eq:diff2_app}
\|\eta^+u_m\|^2-\|\eta^-u_m\|^2=\|\eta^-u_{m+2}\|^2-\|\eta^+u_{m-2}\|^2.
\end{equation}
Summing \eqref{eq:diff2_app} for $m=3,4,\dots,M$ gives the telescoping identity
\begin{equation}\label{eq:telescoping_app}
\sum_{m=3}^{M}\big(\|\eta^+u_m\|^2-\|\eta^-u_m\|^2\big)
=\|\eta^-u_{M+2}\|^2+\|\eta^-u_{M+1}\|^2-\|\eta^+u_1\|^2-\|\eta^+u_2\|^2.
\end{equation}
By \eqref{eq:energy_identity_mode} and $K\le -\kappa_0$,
\begin{equation}\label{eq:telescoping_bound_app}
\sum_{m=3}^{M}\big(\|\eta^+u_m\|^2-\|\eta^-u_m\|^2\big)
=\sum_{m=3}^{M}\frac{m}{2}\int K|u_m|^2\,d\mu
\le -\frac{\kappa_0}{2}\sum_{m=3}^{M}m\,\|u_m\|^2.
\end{equation}
Since $\eta^-u\in L^2(SX)$ and $\eta^-u=\sum_m \eta^-u_m$ with orthogonal summands (because $\eta^-u_m\in\mathcal{H}_{m-1}$), we have $\|\eta^-u_m\|\to 0$ as $|m|\to\infty$, hence letting $M\to\infty$ in \eqref{eq:telescoping_app}-\eqref{eq:telescoping_bound_app} yields
\[
-\|\eta^+u_1\|^2-\|\eta^+u_2\|^2
\le -\frac{\kappa_0}{2}\sum_{m\ge 3}m\,\|u_m\|^2.
\]
Both sides are $\le 0$, so necessarily $\sum_{m\ge 3}m\,\|u_m\|^2=0$, hence $u_m\equiv 0$ for all $m\ge 3$.  Applying the same argument to the negative tail (using \eqref{eq:tail_relation} for $m\le -3$) gives $u_m\equiv 0$ for all $m\le -3$.  Together with \eqref{eq:odd_modes} we conclude
\begin{equation}\label{eq:mode_reduction_final}
u=u_{-1}+u_1,\qquad u_{\pm 1}\in\mathcal{H}_{\pm 1}.
\end{equation}

\smallskip

\noindent\textbf{Elliptic regularity and fiber-linearity.}
With \eqref{eq:mode_reduction_final}, the $m=\pm 2$ components of \eqref{eq:mode_equation} read
\begin{equation}\label{eq:solve_u1}
f_2=\eta^+u_1,\qquad f_{-2}=\eta^-u_{-1}.
\end{equation}
The operators $\eta^+:\mathcal{H}_1\to\mathcal{H}_2$ and $\eta^-:\mathcal{H}_{-1}\to\mathcal{H}_{-2}$ are first-order elliptic (their principal symbols are nonvanishing complex combinations of the horizontal symbols of $X$ and $X_\perp$), hence by elliptic regularity $f_{\pm2}\in C^\infty$ implies $u_{\pm1}\in C^\infty$.  Finally, $\mathcal{H}_{\pm1}$ consists precisely of fiber-linear functions: in a local oriented orthonormal frame $(e_1,e_2)$ and fiber angle $\theta$,
\[
u_1(x,\theta)=a(x)e^{i\theta},\qquad u_{-1}(x,\theta)=\overline{a(x)}e^{-i\theta},
\]
so $u=u_{-1}+u_1=v^1(x)\cos\theta+v^2(x)\sin\theta$ with a unique smooth vector field $v=v^1e_1+v^2e_2$ on $X$.  Equivalently,
\begin{equation}\label{eq:u_fiber_linear}
u(x,\xi)=\langle \xi,\ v(x)\rangle_g\qquad\text{on }S^*_gX.
\end{equation}
This is the $\pm 1$-mode reduction used in \eqref{eq:metric_variation_is_lie}.


\end{document}